\documentclass[12pt]{article}

\usepackage[inner=32mm,outer=31mm,tmargin=30mm,bmargin=40mm]{geometry}

\usepackage{latexsym,amsfonts,amsmath,amssymb,epsfig,tabularx,amsthm,dsfont,mathrsfs}

\usepackage{graphicx}
\usepackage{enumerate}

\usepackage{booktabs}
\usepackage{titling}
\newcommand{\subtitle}[1]{%
  \posttitle{%
    \par\end{center}
    \begin{center}\large#1\end{center}
    \vskip0.5em}%
}

\usepackage{hyperref} 
\hypersetup{colorlinks=true,
        linkcolor=black,
        citecolor=black,
        urlcolor=blue}
\usepackage{multirow}

\usepackage{pgfplots}

\usetikzlibrary{external}
\theoremstyle{plain}

\newtheorem{thm}{Theorem}[section]
\newtheorem{lem}[thm]{Lemma}

\newtheorem{cor}[thm]{Corollary}

\theoremstyle{definition}
\newtheorem{rem}[thm]{Remark}

\renewcommand{\P}{{\mathbb P}}
\newcommand{\expect}{\operatorname{\mathbb{E}}}

\DeclareMathOperator{\Uniform}{Unif}

\newcommand{\dint}{\,\mathup{d}}

\newcommand{\Boxes}{\mathcal B} 
\DeclareMathOperator{\dispersion}{disp}

\newcommand{\eps}{\varepsilon}
\newcommand{\euler}{\textup{e}}

\renewcommand{\rho}{\varrho}

\newcommand{\Q}{{\mathbb Q}}

\newcommand{\N}{{\mathbb N}}

\DeclareMathAlphabet{\mathup}{OT1}{\familydefault}{m}{n}

\newcommand{\wt}{\widetilde}
\newcommand{\widebar}[1]{\mbox{\kern1.5pt\hbox{\vbox{\hrule height 0.6pt \kern0.35ex
        \hbox{\kern-0.15em \ensuremath{#1 }\kern0.0em}}}}\kern-0.1pt}

\newlength{\fixboxwidth}
\usepackage{color}
\usepackage{soul}

\usepackage{latexsym,amsfonts,amsmath,amssymb,mathrsfs}
\usepackage{url}
\usepackage{graphicx}
\usepackage{color}
\usepackage{euscript}
\usepackage{verbatim}
\usepackage{hyperref}

\definecolor{owngreen}{rgb}{0, 0.7, 0.2}

\begin{document}

\title{Expected dispersion of uniformly distributed points}

\author{Aicke Hinrichs\thanks{Institut f\"ur Analysis, 
Johannes Kepler Universit\"at Linz, 
Altenbergerstrasse~69, 4040~Linz, Austria, 
Email: \mbox{aicke.hinrichs@jku.at}, \mbox{david.krieg@jku.at}. 
},
David Krieg$^\ast$,
Robert J. Kunsch\thanks{Lehrstuhl f\"ur Mathematik der Informationsverarbeitung, Pontdriesch 10, 52062 Aachen, Germany, Email: \mbox{kunsch@mathc.rwth-aachen.de}}, 
Daniel Rudolf\thanks{Institute for Mathematical Stochastics, 
Georg-August-Universit\"at G\"ottingen, Goldschmidtstra\ss e 7, 37077 G\"ottingen, 
Email: daniel.rudolf@uni-goettingen.de}}

\date{\today}

\maketitle
\begin{abstract}
  The dispersion of a point set in $[0,1]^d$
  is the volume of the largest axis parallel box inside the unit cube
  that does not intersect with the point set. We study the expected dispersion 
  with respect to a random set of $n$ points
  determined by an i.i.d.\ sequence of uniformly distributed random variables.
  Depending on the number of points $n$ and the dimension $d$ 
  we provide an upper and lower bound of the expected dispersion.
  In particular, we show that the minimal number of points
  required to achieve an expected dispersion less than $\varepsilon\in(0,1)$
  depends linearly on the dimension $d$.
\end{abstract}


{\bf Classification.
  Primary:
    62D05;
  Secondary:
    52B55,
    65Y20,
    68Q25.}

\section{Introduction and main result}
Given 
$n$ points $\{x_1,\dots,x_n\}\subset [0,1]^d$,
the dispersion is the volume of the largest 
axis parallel box
that does not contain a point.
It is defined by
\begin{equation}  \label{eq: def_disp}
  \dispersion(x_1,\dots,x_n) := \sup_{B\cap \{x_1,\dots,x_n\} = \emptyset} \lambda_d(B),
\end{equation}
where $\lambda_d$ denotes the $d$-dimensional Lebesgue measure and 
the supremum is taken over all boxes  $B=I_1\times \dots \times I_d$
with intervals $I_k\subseteq [0,1]$.
In this note we study the expected dispersion of random points
based on an i.i.d.\ sequence of uniformly distributed random variables $(X_i)_{i\in\N}$,
where each $X_i$ maps
from a common probability space $(\Omega,\mathcal{F},\P)$ to $[0,1]^d$.
For simplicity we write
$X_1,X_2,\ldots 
  \overset{\textup{iid}}{\sim} \Uniform([0,1]^d)$.
We ask for the behavior of
\begin{equation*}
  \expect(\dispersion(X_1,\dots,X_n))
\end{equation*}
in terms of $n$ and $d$.

In recent years the proof of existence and the construction
of point sets with small dispersion
attracted considerable attention, see
 \cite{AiHiRu15,Kr18,Ru18,Sos18,UllVyb18,UllVyb19}.
In order to describe optimality
of such point sets of cardinality $n$
in the $d$-dimensional setting,
let us define the minimal dispersion
\begin{equation*}
  \dispersion(n,d) := \inf_{\{x_1,\dots,x_n\} \subset [0,1]^d} \dispersion(x_1,\dots,x_n),
\end{equation*}   
and its inverse 
\begin{equation*}
 n(\eps,d) := \min\{n\in\N\mid \dispersion(n,d) \leq \eps\}, 
\end{equation*}
where $\eps \in(0,1)$. 
A lower bound for the minimal dispersion growing with the dimension $d$
is provided in \cite[Theorem~1]{AiHiRu15}.
Moreover, \cite[Section~4]{AiHiRu15} contains an upper bound due to Gerhard Larcher,
based on constructions of digital nets,
which give explicitly constructable point sets.
For $\eps\in(0,1/8)$ the bounds are 
\begin{equation} \label{eq: basic_lower_and_upper_bnd}
    2^{-3}\eps^{-1} \log_2 d \leq n(\eps,d) \leq 2^{7d+1} \eps^{-1}.
\end{equation}
Clearly, the dependence on $\eps^{-1}$ in \eqref{eq: basic_lower_and_upper_bnd} cannot be improved.
However, the upper bound grows exponentially in $d$, while the lower bound only grows logarithmically.
For large dimensions the upper bound can be improved significantly.
It is shown in \cite{Sos18,UllVyb18} that for fixed $\eps$
the quantity $n(\eps,d)$ increases at most logarithmically in~$d$.
This means that also the $d$-dependence of the lower bound in
\eqref{eq: basic_lower_and_upper_bnd} is optimal.

The results of \cite{Sos18,UllVyb18} are based on probabilistic arguments.
Namely, points are drawn uniformly at random from a regular grid
whose parameters depend on $\eps$ and $d$,
and it is shown that these points are suitable with positive probability.
By the use of a derandomization technique,
\cite{UllVyb19} provides a deterministic algorithm
for the construction of point sets with cardinality $c_\eps \log_2(d)$
and dispersion at most $\eps$,
where $c_\eps>0$ depends only polynomially on $\eps$.
Comparable results can be obtained via a careful translation
of statements on the so-called hitting set problem, see \cite{linial1997efficient}.
In Table~\ref{table: comparison} we survey explicit bounds
for $n(\eps,d)$,
in particular, it contains the ones of \cite{Sos18,UllVyb18}
and their dependence on $\eps$.  
\begin{table} 
\begin{center}
\begin{tabular}{lll} 
\toprule
Reference
  & Upper bound of $n(\eps,d)$
    & Remarks \\ 
\midrule 
\cite{AiHiRu15}
  & \multirow{2}{*}{$\lceil2^{7d+1}\,\eps^{-1}\rceil$}
    & optimal in $\eps^{-1}$\\
(due to Larcher)
  & & digital net construction
    \\
\midrule
\cite{Kr18}
  &     $\min\{(2d)^{\lceil \log_2(\eps^{-1})-1 \rceil}\,,\,
                \eps^{-1}\lceil \log_2(\eps^{-1}) \rceil^{d-1}\}$
    &  sparse grid construction
      \\     
\midrule
\cite{Ru18}
  & $8\,d\,\eps^{-1}\log(33\,\eps^{-1})$
    &  existence of point set \\                  
\midrule
\multirow{2}{*}{\cite{Sos18}}
  & \multirow{2}{*}{
    $\log_2(d)
            \,\left\lceil\eps^{-1}\right\rceil^{(\lceil\eps^{-1}\rceil^2+2)}
            \,(4 \log\lceil\eps^{-1}\rceil+1)$}
    & optimal in $d$\\
  & & existence of point set\\
\midrule
\multirow{2}{*}{\cite{UllVyb18}}
  & \multirow{2}{*}{$2^7\log_2(d)\,\eps^{-2}(1+\log_2(\eps^{-1}))^2$}
    & optimal in $d$\\
  & & existence of point set\\ 
    \bottomrule
  \end{tabular}  
  \end{center}
\label{table: comparison}
\caption{The table contains several upper bounds on $n(\eps,d)$ based on existence results of ``good'' points as well as explicit constructions.}
\end{table}

In one way or another, most of these upper bounds
rely on randomly drawn points and probabilistic arguments.
In particular, the estimate of \cite[Corollary~1]{Ru18}
is based on an i.i.d.\ sequence of
random variables uniformly distributed on $[0,1]^d$.
Maybe this is the most canonical randomly chosen point set
and one might ask how good it is compared to deterministic point sets.
Here, the measure of goodness is the expected dispersion
and our main result
reads as follows:
\begin{thm} \label{thm: main_result}
For any $n> d$ we have
\begin{equation*}
  \max\left\{\frac{\log(n)}{9\, n},\frac{d}{2\euler\, n}\right\}
    \leq \expect(\dispersion(X_1,\dots,X_n))
    \leq \frac{9d}{n} \log\left(\frac{\euler\, n}{d}\right).
\end{equation*}
\end{thm}

Let us also state our result in terms of the inverse of the expected dispersion.
For $\eps\in(0,1)$ and $d\in\N$,
the inverse of the expected dispersion is defined as
\begin{equation*}
  N(\eps,d)
    := \min\{n\in\N\mid \expect\dispersion(X_1,\dots,X_n) \leq \eps \}.
\end{equation*}

\begin{cor} \label{cor: main_result}
For all $\eps \in (0,\frac{1}{9\euler})$
and $d\in \N$ we have
\begin{equation*}
  \max\left\{\frac{1}{9\eps} \log\left(\frac{1}{9\eps}\right) ,\,
    \frac{d}{2\euler\,\eps}\right\}
    \leq N(\eps,d)
    \leq \left\lceil 9 (1+\euler^{-1})\,\frac{d}{\eps}
                       \log\left(\frac{9(\euler + 1)}{\eps}\right)
         \right\rceil. 
\end{equation*}
\end{cor}

These estimates show that
$N(\eps,d)$ for fixed $\eps$
behaves linearly w.r.t.\ the dimension,
and for fixed $d$ behaves like $\eps^{-1}\log(\eps^{-1})$.
It is interesting to note that the linear behavior w.r.t.\ $d$
is in contrast to the $\log_2(d)$ dependence of the inverse of the minimal dispersion. 

The upper bound of Theorem~\ref{thm: main_result} follows 
by exploiting a $\delta$-cover approximation and a concentration inequality
stated in \cite{Ru18}.
The proof of the lower bound is separated into two parts.
First, we derive the bound $\log(n)/(9n)$
from well known results on the coupon collector's problem.
After that the $d$-dependent lower bound $d/(2\euler\,n)$
is proven by a 
reduction to the expected dispersion of $d$ points
and, eventually, a constant lower bound for this quantity.

The proof of Theorem~\ref{thm: main_result},
  along with the necessary notation,
is given in Section~\ref{sec: proof}.
Further discussions and extensions of the results
are provided in Section~\ref{sec: notes_rem}.

\section{Proof of Theorem~\ref{thm: main_result}}
\label{sec: proof}
\subsection{The upper bound}
Before we start with the proof of the upper bound
let us provide some further notation.
Let $\Boxes$ be the set of boxes given as follows,
\begin{equation*}
  \Boxes
    := \left\{ \prod_{k=1}^d [a^{(k)},b^{(k)}) \subseteq [0,1]^d \mid 
         a^{(k)},b^{(k)}\in \Q\cap [0,1], k=1,\dots,d
       \right\}.
\end{equation*}
Then, obviously, we have
\begin{equation*}
  \dispersion(x_1,\dots,x_n)
    = \sup_{\substack{B\in\Boxes \\ B \cap \{x_1,\dots,x_n\} = \emptyset}}
        \lambda_d(B).
\end{equation*}
Note that with this we can restrict ourself to boxes determined by half-open intervals
with rational boundary values.
Thus, the supremum within the dispersion is only taken over a countable set,
which leads to the measurability of the mapping
$(x_1,\dots,x_n)\mapsto \dispersion(x_1,\dots,x_n)$.
Occasionally, we also call $\Boxes$ the set of test sets.
Let $\delta\in(0,1]$, 
  then a \emph{$\delta$-cover}
of the set of test sets $\Boxes$ 
is given by a finite set $\Gamma_\delta\subset \Boxes$ that
satisfies 
\begin{equation*}
  \forall B\in \Boxes \quad\exists L_B, U_B\in \Gamma_\delta
  \quad \text{with} \quad L_B\subseteq B \subseteq U_B
  \quad \text{and} \quad \lambda(U_B\setminus L_B) \leq \delta.
\end{equation*} 
Furthermore,
for $x_1,\dots,x_n\in[0,1]^d$ and a $\delta$-cover $\Gamma_\delta$ for $\Boxes$
define
\begin{equation*}
  \dispersion_{\delta} (x_1,\dots,x_n)
    := \sup_{\substack{A\in \Gamma_\delta \\
              A\cap\{x_1,\dots,x_n\}=\emptyset}}
         \lambda_d(A).  
\end{equation*} 
Having introduced those quantities we state two results from \cite{Ru18}. 
From $\Gamma_\delta$ being a $\delta$-cover
it follows that
\begin{equation} \label{eq: disp_delta_cover}
  \dispersion(x_1,\dots,x_n) \leq \delta + \dispersion_{\delta}(x_1,\dots,x_n),
\end{equation}
and, 
  via a union bound,
it follows that for any $s\in(0,1)$ we have
\begin{equation} \label{eq: concentration}
  \P\left(\dispersion_\delta(X_1,\dots,X_n)>s\right)
    \leq |\Gamma_\delta| (1-s)^n.
\end{equation}
We refer to \cite[Lemma~1]{Ru18} and the proof of \cite[Theorem~1]{Ru18} for details.
These
results lead to the following lemma.
\begin{lem}  \label{lem: upp_bnd_delta_cover}
  Let $\delta\in(0,1]$ and assume that the set $\Gamma_\delta$ is a $\delta$-cover of $\Boxes$. 
  Then, for any $n\geq \log | \Gamma_\delta|$ we have
  \begin{equation*}
    \expect(\dispersion(X_1,\dots,X_n))
      \leq \delta + \frac{\log| \Gamma_\delta |}{n} + \frac{1}{n+1}.
  \end{equation*} 
\end{lem}
\begin{proof}
  From \eqref{eq: disp_delta_cover} we have
  \begin{equation*}
    \expect(\dispersion(X_1,\dots,X_n))
      \leq \delta + \expect(\dispersion_\delta(X_1,\dots,X_n)).
  \end{equation*}
  Furthermore, by using \eqref{eq: concentration} we obtain
  \begin{align*}
    \expect&(\dispersion_\delta(X_1,\dots,X_n))
      = \int_0^{1} \P\left(\dispersion_\delta(X_1,\dots,X_n) > s\right) \dint s \\
      &\leq \frac{\log|\Gamma_\delta|}{n}
              + \int_{(\log|\Gamma_\delta|)/n}^1
                    \P\left( \dispersion_\delta(X_1,\dots,X_n) > s \right)
                  \dint s\\
    & \leq \frac{\log|\Gamma_\delta|}{n}
            + |\Gamma_\delta| \int_{(\log|\Gamma_\delta|)/n}^1 (1-s)^n \dint s
      \leq \frac{\log|\Gamma_\delta|}{n}
            + \frac{|\Gamma_\delta|}{n+1}
                \left(1-\frac{\log|\Gamma_\delta|}{n}\right)^{n+1}.
  \end{align*}  
  Note that for any $0\leq a\leq n$ we have $(1-a/n)^n\leq \exp(-a)$,
  hence,
  \begin{equation*}
    | \Gamma_\delta| \left(1-\frac{\log|\Gamma_\delta|}{n} \right)^{n+1} \leq 1,
  \end{equation*}
  which finishes the proof.
\end{proof}

\begin{rem}
  Except for the assumption that we have a $\delta$-cover,
  we did not use any property of the set of test sets $\Boxes$.
\end{rem}
Now,
the upper bound of Theorem~\ref{thm: main_result} is deduced by the results
on $\delta$-covers for $\Boxes$ from Gnewuch, see \cite{Gn08}.
Namely,
from \cite[Formula (1), Theorem 1.15 with $d! \geq (d/\euler)^d$, and Lemma 1.18]{Gn08}
one obtains that
there is a $\delta$-cover for $\Boxes$ with
  $|\Gamma_\delta| \leq (6\euler\, \delta^{-1})^{2d}$.
By setting $\delta = 6d/n$,
the upper bound of Theorem~\ref{thm: main_result} follows
with Lemma~\ref{lem: upp_bnd_delta_cover} and $n\geq d$.

Finally, this upper estimate implies the upper bound of Corollary~\ref{cor: main_result}.
For the convenience of the reader, we add 
a few arguments. 
We are looking for
preferably small integers $n \geq d$
such that
\begin{equation*}
  \frac{9d}{n} \log\left(\frac{\euler\, n}{d}\right) \leq \eps \,.
\end{equation*}
Note that the term on the left-hand side of this inequality
  is monotonically decreasing for~$n \geq d$.
  Let $c \geq 1$ be a constant,
  then, any integer
\begin{equation*}
  n \geq c \, \frac{d}{\eps} \log\left(\frac{c\, \euler}{\eps}\right) 
\end{equation*}
  will also comply with $n \geq d$, and hence satisfies the estimate
\begin{equation*}
  \frac{9d}{n} \log\left(\frac{\euler\, n}{d}\right)
    \leq \frac{9}{c} \, \eps
       \cdot \left(1 + \frac{\log \log\left(\frac{c\,\euler}{\eps}\right)
                             }{\log\left(\frac{c\,\euler}{\eps}\right)}
             \right)
    \leq \frac{9(1+\euler^{-1})}{c} \, \eps \,,
\end{equation*}
where we used that $(\log x)/x$ attains its maximum for~$x = \euler$.
Choosing the constant $c = 9(1+\euler^{-1}) = 12.31...$
we obtain the desired guarantee.

\subsection{The lower bound}

In Section~\ref{sec: low_bnd_n} we show that
$\expect(\dispersion(X_1,\dots,X_n)) \geq \frac{\log(n)}{9 n}$,
and in Section~\ref{sec: low_bnd_d} we prove that 
$\expect(\dispersion(X_1,\dots,X_n)) \geq \frac{d}{2\euler \, n}$ for $n>d$.
Both lower bounds together yield the corresponding statement
of Theorem~\ref{thm: main_result}.
By convention, all random variables are defined on a common probability space
$(\Omega,\mathcal{F},\P)$.

\subsubsection{Lower bound without dimension dependence}
\label{sec: low_bnd_n}

We start with 
an auxiliary tool,
using results on the \emph{coupon collector's problem}.
\begin{lem} \label{lem: low_bnd_lem_n}
  For $\ell\in \N$ let $(Y_i)_{i\in\N}$ be an i.i.d.\ sequence
  of uniformly distributed random variables in $\{1,\dots,\ell\}$. 
  Define $H_\ell := \sum_{j=1}^\ell j^{-1}$ and
  \begin{equation*}
    \tau_{\ell} := \min\{k\in\N\mid \{Y_1,\dots,Y_k\}= \{1,\dots,\ell\}\}.
  \end{equation*}  
  Then, for any integer $n\leq (H_{\ell} -2)\ell$ 
  we have $\P(\tau_\ell>n) > 1/2$. 
\end{lem}
\begin{proof}
  It is well known that the mean and the variance of $\tau_\ell$ satisfy
  \begin{equation*}
    \expect\tau_\ell = \ell H_\ell 
    \quad \text{and} \quad
    \mathrm{Var}\, \tau_\ell \leq \ell^2 \sum_{j=1}^\ell j^{-2}
    \leq \frac{\pi^2}{6} \ell^2 \,.
  \end{equation*}
  For details concerning these estimates,
  see for example \cite{LePeWi09} or \cite[Proposition~4.7]{Kr19Phd}.
  Then, for $n \leq (H_\ell - 2) \ell$, by Chebyshev's inequality we have
  \begin{align*}
    \P(\tau_\ell \leq n)
      \leq \P(\tau_\ell \leq (H_\ell - 2) \ell)
      = \P(\ell H_\ell - \tau_\ell \geq 2\ell)
      \leq \frac{\mathrm{Var}(\tau_\ell)}{4 \ell^2}
      \leq \frac{\pi^2}{24} 
      < \frac{1}{2},
  \end{align*}
  which finishes the proof.
\end{proof}
By means of the previous result
we are able to prove the desired lower bound in the following lemma.
\begin{lem} \label{lem:dim_indep_LB}
  For any integer $n \geq 3$ 
  we have $\expect(\dispersion(X_1,\dots,X_n)) > \frac{\log(n)}{9 \, n}$.
\end{lem}
\begin{proof}
  Fix some $\ell \in\N$ and
  split $[0,1]^d$ into $\ell$ disjoint boxes $B_1,\dots,B_\ell$
  of equal volume~$1/\ell$
  (e.g.\ split along the first coordinate).
   For $i=1,\ldots,n$ define the random variable $Y_i: \Omega \to \{1,\ldots,\ell\}$
    that indicates the box the point $X_i$ lies in,
    i.e.\ \mbox{$X_i(\omega) \in B_{Y_i(\omega)}$}.
  Note that $Y_1,\dots,Y_n$ are i.i.d.\ 
  and each uniformly distributed in $\{1,\dots,\ell\}$.
  Furthermore, for $\omega\in \Omega$ satisfying 
  \begin{equation*}
      \{ Y_1(\omega),\dots,Y_n(\omega) \} \not = \{1,\dots,\ell\} \,,
  \end{equation*}
  there is an index $r\in \{1,\dots,\ell\}$ such that
  $
    \{X_1(\omega),\dots,X_n(\omega)\} \cap B_r =\emptyset.
  $
  Thus, for such an $\omega$ we obtain
  \begin{equation*}
    \dispersion(X_1(\omega),\dots,X_n(\omega)) \geq 1/\ell.
  \end{equation*}
  This yields
  \begin{align*}
    \expect (\dispersion (X_1,\dots,X_n))
    & = \int_\Omega \dispersion(X_1(\omega),\dots,X_n(\omega)) \P(\dint \omega) \\
    & \geq \frac{1}{\ell} \; \P(\{Y_1,\dots,Y_n\}\not = \{1,\dots,\ell\}).
   \end{align*}
   Observe that with $\tau_\ell$ defined in Lemma~\ref{lem: low_bnd_lem_n} we have
   \begin{equation*}
     \P(\{Y_1,\dots,Y_n\}\not = \{1,\dots,\ell\})
       = \P(\tau_\ell > n).
   \end{equation*}
  Choosing $\ell := \left\lceil \frac{(1+\euler)\, n}{\log(n)} \right\rceil$,
   we get
   \begin{align*}
     \frac{n}{\ell}
      \leq \frac{\log(n)}{1 + \euler} 
      \leq \log\left(\frac{(1 + \euler)\,n}{\log(n)}\right)-2
      \leq \log(\ell) - 2
      < H_\ell-2,
   \end{align*}
   where we used the inequality
     $\log\left(\frac{(1 + \euler)\,x}{\log(x)}\right) - 2 - \frac{\log(x)}{1+\euler}
       \geq 0$ for $x > 1$
    (attaining equality in~$x = \exp(1+1/\euler)$),
    as well as $H_\ell = \sum_{j=1}^{\ell} j^{-1} > \log (\ell+1)$.
  This asserts $n \leq (H_\ell - 2)\ell$, 
  and by Lemma~\ref{lem: low_bnd_lem_n} we obtain
  $\P(\tau_\ell >n) > 1/2$. 
  Taking everything together yields
  \begin{align*}
    \expect (\dispersion (X_1,\dots,X_n))
      > \frac{1}{2\ell} \geq \frac{1}{2} \cdot \frac{\log(n)}{(1+\euler)\,n+\log(n)}
      > \frac{\log(n)}{9\,n},
  \end{align*}
 which completes the proof.
 Our derivation holds for integers $n \geq 2$,
   but the bound starts decaying for~$n \geq 3$, in the first place.
\end{proof}  
Having the result of the previous lemma,
the first part
within the maximum of the lower bound in Corollary~\ref{cor: main_result} follows.
For the convenience of the reader we add a few arguments.
If the expected dispersion shall be smaller than a given~$\eps > 0$,
  the number of points, $n$, must satisfy~$\frac{\log n}{9n} \leq \eps$.
  Note that the left-hand side is monotonically 
  decreasing only for~$n \geq \euler$,
  but the expected dispersion for \mbox{$n \in \{1,2\}$}
  should be larger or equal the expected dispersion for~$n=3$.
  Restricting to \mbox{$\eps \in (0,\frac{1}{9\euler})$},
  for $\euler \leq n < \frac{1}{9\eps} \log\left(\frac{1}{9\eps}\right)$ we would have
  \begin{equation*}
    \frac{\log n}{9n}
      > \eps \cdot \left(1 + \frac{\log \log \left(\frac{1}{9\eps}\right)
                                 }{\log \left(\frac{1}{9\eps}\right)}
                   \right)
      > \eps \,.
  \end{equation*}
  Hence, $n \geq \frac{1}{9\eps} \log\left(\frac{1}{9\eps}\right)$ is necessary
  for the expected dispersion to be less or equal~$\eps$.

  \begin{rem}
    In the strong asymptotic regime, the prefactor~$\frac{1}{9}$
    in Lemma~\ref{lem:dim_indep_LB} vanishes, i.e.
    \begin{equation*}
      \expect(\dispersion(X_1,\dots,X_n)) \gtrsim \frac{\log(n)}{n}
      \qquad\text{for $n \to \infty$.}
    \end{equation*}
    This result can be deduced by considering the volume~$V_n$
    of the largest empty box of the shape~$B = (a,b)\times[0,1]^{d-1}$,
    with $0 < a < b < 1$.
    This quantity~$V_n$ is equivalent to the distribution
    of the maximal distance between adjacent points
    when distributing $n+1$ random points on a circle with perimeter~$1$.
    We can use an asymptotic result by Schlemm~\cite[Cor.~1]{Schl14}
    on the largest gap of random points on a circle,
    namely that $(n+1) V_n - \log(n+1)$ converges
    to a standard Gumbel distribution with its
    expectation being  the Euler-Mascheroni constant
    $\gamma \approx 0.5772$.
    Hence,
    \begin{equation*}
      \expect(\dispersion(X_1,\dots,X_n))
        \geq \expect V_n
        \simeq \frac{\gamma + \log(n+1)}{n+1}
        \simeq \frac{\log(n)}{n}
      \qquad\text{for $n \to \infty$.}
    \end{equation*}
\end{rem}

\subsubsection{Dimension-dependent lower bound}
\label{sec: low_bnd_d}

The proof of the lower bound w.r.t.\ the dimension is 
divided into two steps.
First, we deduce a lower bound
for 
the expected dispersion of $n$ points
in terms of the expected dispersion of $d$ points, see Lemma~\ref{lem:reduction_to_n=d}.
Thus, we reduce the problem to finding a lower bound
for 
the expected dispersion of $d$ points,
which then is the goal of the second step, see Lemma~\ref{lem:exp_disp_d_points}.
In the following proof, 
for $B\in\Boxes$ and $x_1,\dots,x_\ell$
with $\ell\in\N$, we use the notation
\begin{equation*}
  \dispersion_{\mid B}(x_1,\dots,x_\ell)
    :=\sup_{\substack{R\in\Boxes\cap B \\
                      R\cap\{x_1,\dots,x_\ell\} = \emptyset}} 
        \lambda_d(R)
\end{equation*}
for the dispersion restricted to $B$. 
The following reduction lemma is a probabilistic version of \cite[Lemma~1]{AiHiRu15}.
\begin{lem}  \label{lem:reduction_to_n=d}
  For any $n,\ell\in \N$ we have
  \begin{equation*}
    \expect(\dispersion(X_1,\dots,X_n)) \geq \frac{\ell+1}{n+\ell+1} \,\expect(\dispersion(X_1,\dots,X_\ell)).
  \end{equation*}
\end{lem}

\begin{proof}
  We start with a purely combinatorial argument,
  a version of the pigeonhole principle.
  If we split $[0,1]^d$ into $m$ boxes $B_1,\ldots,B_m$ of equal volume,
  then there is some $j \in \{1,\ldots,m\}$ such that $B_j$ contains no more
  than $\lfloor n/m \rfloor$ of the points $X_1,\ldots,X_n$.
  Choosing
  $m = \lceil \frac{n+1}{\ell+1}\rceil = \lfloor \frac{n}{\ell + 1} \rfloor + 1$,
  we have
  \mbox{$\lfloor \frac{n}{m} \rfloor
           \leq \lfloor \frac{n}{n+1}(\ell+1) \rfloor
           \leq \ell$}.
  For \mbox{$k\in \N$}, let $n_k\in\N$ be the time when $B_j$ is hit
  by the sequence $(X_i)_{i \in \N} \subset [0,1]^d$
  for the $k$-th time
  (which for \mbox{$X_1,X_2,\ldots \overset{\textup{iid}}{\sim} \Uniform([0,1]^d$)} 
  almost surely happens).
  With $n_\ell\ge n$, by the choice of~$B_j$, we have
  \begin{equation*}
    \dispersion(X_1,\ldots,X_n) 
      \geq \dispersion_{\mid B_j}(\{X_1,\ldots,X_n\}\cap B_j) 
      \geq \dispersion_{\mid B_j}(X_{n_1},\ldots,X_{n_\ell}).
  \end{equation*}
  Let $T$ be 
  an affine transformation
  that maps $B_j$ onto $[0,1]^d$.
  Then
  \begin{equation*}
    \dispersion_{B_j}(X_{n_1},\ldots,X_{n_\ell})
      = \lambda_d(B_j) \cdot \dispersion(TX_{n_1},\ldots,TX_{n_\ell}).
  \end{equation*}
  Recall that $X_1,X_2,\ldots \overset{\textup{iid}}{\sim} \Uniform([0,1]^d)$,
  hence, 
  the points $TX_{n_1},\ldots,TX_{n_\ell}$ are independent
  and uniformly distributed in $[0,1]^d$.
  Taking the expectation and using
  $\lambda_d(B_j)
      = \frac{1}{m}
      \geq \frac{1}{n/(\ell+1) + 1} 
      = \frac{\ell+1}{n+\ell+1}$,
  yields the statement.
\end{proof}

Having the previous lemma at hand,
it is sufficient to provide a constant lower bound
for the expected dispersion of $d$ points.
In a slightly more general way,
we obtain the following.
\begin{lem}  \label{lem:exp_disp_d_points}
  For any $d,\ell\in\N$
  and $X_1,\ldots, X_{\ell} \overset{\textup{iid}}{\sim} \Uniform([0,1]^d)$
  we have
  \begin{equation*}
    \expect(\dispersion(X_1,\dots,X_{\ell})) \geq \euler^{-\ell/d}.
  \end{equation*}
\end{lem}
\begin{proof}
  For all $i \in \{1,\ldots,\ell\}$, 
  let $X_i^{\ast}$ denote the largest coordinate of $X_i$, i.e.,
  \begin{equation*}
    X_i^{\ast} := \max\{ X_i^{(1)},\ldots,X_i^{(d)}\}.
  \end{equation*}
  We choose $j^{\ast}(i)\leq d$ such that $X_i^{(j^{\ast}(i))}=X_i^{\ast}$.
  Let us consider the box 
  \begin{equation*}
    B=\prod_{j=1}^d [0,a_j) \,,
  \end{equation*}
  where 
  \begin{equation*}
    a_j := \min\left(\{1\} \cup \bigl\{X_i^{\ast} \,\big|\,
                                    i\leq \ell \text{ with } j^{\ast}(i) = j
                             \bigr\}
               \right)\,.
  \end{equation*}
  This box is empty, since for all $i\leq \ell$ 
  we have $X_i^{(j^{\ast}(i))} \geq a_{j^{\ast}(i)}$,
  and hence $X_i \not\in B$.
  An illustration for the case $d=2$ 
  is provided in Figure~\ref{fig: no_point}.
  On the other hand, the volume of $B$ is given by
  \begin{equation*}
    \lambda^d(B) = \prod_{j=1}^d a_j = \prod_{i \in I} X_i^{\ast} \,,
  \end{equation*}
  where $I$ is 
  a suitable subset of $\{1,\ldots,\ell\}$.
  This yields
  \begin{equation*}
    \dispersion(X_1,\dots,X_{\ell})
      \geq \prod_{i \in I} X_i^{\ast} \geq \prod_{i=1}^{\ell} X_i^{\ast}.
  \end{equation*}
  The random numbers $X_i^{\ast}$ are independent and beta distributed
  with parameters $\alpha=d$ and $\beta=1$, 
  in particular, $\expect(X_i^{\ast})=1-1/(d+1)$. Hence, 
  \begin{align*}
    \expect(\dispersion(X_1,\dots,X_{\ell}))
      &\geq \prod_{i=1}^{\ell} \expect( X_i^{\ast})
      = \left(1 - \frac{1}{d+1}\right)^{\ell}
      = \left(\frac{1}{1+\frac{1}{d}}\right)^{\ell} \\
      &\geq \left(\frac{1}{\exp(1/d)}\right)^{\ell}
      = \euler^{-\ell/d} \,.
   \qedhere
 \end{align*}
\end{proof}
The proof of the lower bound follows by setting $\ell=d$
and combining the results of the two lemmas.
We readily get
\begin{equation*}
  \expect(\dispersion(X_1,\dots,X_n)) \geq \frac{d+1}{\euler \,(n+d+1)}
    > \frac{d}{2\euler \, n} \,,
\end{equation*}
where the last inequality follows from $n>d$.
 For $\eps \in (0,\frac{1}{2\euler})$,
  the respective inverse lower bound \mbox{$N(\eps,d) \geq \frac{d}{2\euler\,\eps}$}
  is straightforward, where the restriction on~$\eps$ implies~$N(\eps,d) > d$.
\begin{figure} 
  \begin{minipage}{0.45\textwidth}
    \vspace*{3.8ex}
    \includegraphics[width=\textwidth]{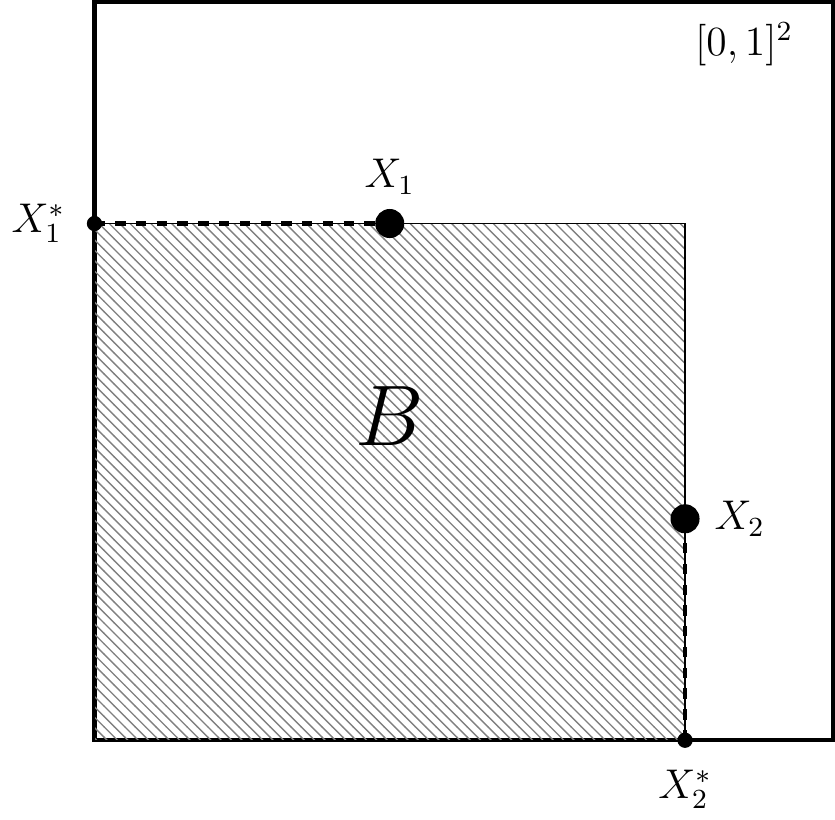}
  \end{minipage}
  \hfill
  \begin{minipage}{0.45\textwidth}
    \includegraphics[width=\textwidth]{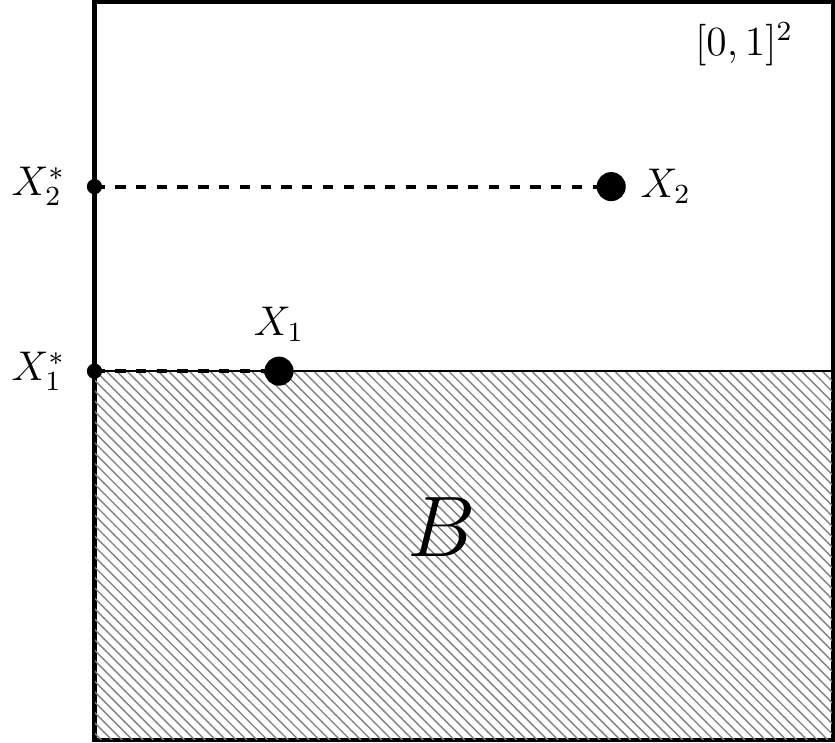}
  \end{minipage}
\caption{An illustration of the empty box construction
  from Lemma~\ref{lem:exp_disp_d_points}
  for $d=\ell=2$, featuring two different types of situations.
  In the left picture we have $B=[0,a_1)\times[0,a_2)$
  with $X_1=(0.4,0.7)$, $j^*(1)=2$, $a_1=0.7$
  and $X_2=(0.8,0.3)$, $j^*(2)=1$, $a_2=0.8$.
  In the right picture we have $B=[0,a_1)\times[0,a_2)$
  with $X_1=(0.25,0.5)$, $j^*(1)=2$, $a_1=0.5$
  and $X_2=(0.7,0.75)$, $j^*(1)=2$, $a_2=1$.}
\label{fig: no_point}
\end{figure}

\section{Notes and remarks} 
\label{sec: notes_rem}
The dispersion of a point set, as defined in \eqref{eq: def_disp},
has been introduced in \cite{RoTi96}, generalizing the work of \cite{Hl76}.
The renewed interest in this quantity emerged from its appearance
in the construction of algorithms for the approximation of rank-one tensors,
see \cite{BaDaDeGr14,KrRu18,NoRu16},
where the dependence on the dimension is crucial. 
It is also related to the universal discretization problem,
see \cite{Temla18JoC},
and the fixed volume discrepancy, see \cite{Temla17,TemlaUll19}.

The dispersion of a point set has also been studied on the torus
instead of the unit cube, see for example \cite{BreHi19,Ul15}.
This setting can be described on 
the unit cube by choosing 
another set of test sets, namely 
\begin{equation*}
  \wt{\Boxes} := \left\{ \prod_{k=1}^d I_k(x,y) \mid x=(x^{(1)},\dots,x^{(d)}),y=(y^{(1)},\dots,y^{(d)})\in [0,1]^d\cap \Q^d \right\},
\end{equation*}
with 
\begin{equation*}
I_k(x,y) = \begin{cases}
(x^{(k)},y^{(k)}) & x^{(k)}<y^{(k)},\\
[0,1]\setminus [y^{(k)},x^{(k)}] & y^{(k)}\leq x^{(k)}.
\end{cases}
\end{equation*}
The set $\wt{\Boxes}$ is called the test set of
\emph{periodic boxes}. Since the proof of the upper bound of the expected dispersion
depends on $\Boxes$ only through the $\delta$-cover,
with the same arguments 
we can also derive      
an upper bound 
for the expected dispersion w.r.t.\ $\wt{\Boxes}$
by using an appropriate \emph{periodic} 
 $\delta$-cover.
For $x_1,\dots,x_n\in [0,1]^d$ define
\begin{equation*}
\widetilde{\dispersion}(x_1,\dots,x_n) := \sup_{\substack{B\in\widetilde{\Boxes} \\ B \cap \{x_1,\dots,x_n\} = \emptyset}}
\lambda_d(B).
\end{equation*}
With \cite[Lemma~2]{Ru18},
we obtain that there is a $\delta$-cover
$\widetilde{\Gamma}_{\delta}$ of $\widetilde{\Boxes}$
  with cardinality
  $|\widetilde{\Gamma}_{\delta}| \leq \left(4d\delta^{-1}\right)^{2d}$,
so that with $\delta = 2d/n$ we have
\begin{equation*}
\expect(\widetilde{\dispersion}(X_1,\dots,X_n)) \leq \frac{5d}{n} \log(2n).
\end{equation*}
By the fact that $\Boxes \subset \widetilde{\Boxes}$ we obtain for any $x_1,\dots,x_n\in [0,1]^d$ that 
\begin{equation*}
\dispersion(x_1,\dots,x_n) \leq \widetilde{\dispersion}(x_1,\dots,x_n),
\end{equation*}
hence, the lower bounds of Theorem~\ref{thm: main_result} also carry over to 
$\expect(\widetilde{\dispersion}(X_1,\dots,X_n))$.
Here it is worth mentioning that the lower bound w.r.t.\ the dimension
can also be deduced from \cite[Theorem~1]{Ul15}. 
Thus, in this setting also
a linear dimension-dependence is present in $\expect(\widetilde{\dispersion}(X_1,\dots,X_n))$. 
However, 
concerning the inverse of the expected dispersion in the periodic case,
the precise growth 
w.r.t.\ the dimension remains open,
we only know that 
it is between $d$ and $d\log(d)$.

\section*{Acknowledgements}

We thank Mario Ullrich and the referees for their valuable suggestions. 
Parts of the work have been done at the Dagstuhl seminar ``Algorithms and Complexity for Continuous Problems''
in August 2019, where we enjoyed the stimulating research environment. 
A.~Hinrichs and D.~Krieg are supported by the Austrian Science Fund (FWF) Project F5513-N26, which is a part of the Special Research Program ``Quasi-Monte Carlo Methods: Theory and Applications''.
Daniel Rudolf gratefully acknowledges support of the 
Felix-Bernstein-Institute for Mathematical Statistics in the Biosciences
(Volks\-wagen Foundation) and the Campus laboratory AIMS.

%

\providecommand{\bysame}{\leavevmode\hbox to3em{\hrulefill}\thinspace}
\providecommand{\MR}{\relax\ifhmode\unskip\space\fi MR }
\providecommand{\MRhref}[2]{%
	\href{http://www.ams.org/mathscinet-getitem?mr=#1}{#2}
}
\providecommand{\href}[2]{#2}

\end{document}